\documentclass[a4paper]{amsart}

\usepackage{amsthm}
\usepackage{amssymb}
\usepackage[all]{xy}
\usepackage{hyperref}
\usepackage{verbatim}

\SelectTips{cm}{}
\calclayout 

\numberwithin{equation}{section}

\theoremstyle{plain}
\newtheorem{lemma}{Lemma}[section]
\newtheorem{theorem}[lemma]{Theorem}
\newtheorem{proposition}[lemma]{Proposition}
\newtheorem{corollary}[lemma]{Corollary}

\theoremstyle{definition}
\newtheorem{definition}[lemma]{Definition}
\newtheorem{example}[lemma]{Example}

\theoremstyle{remark} 
\newtheorem{remark}[lemma]{Remark} 

\newcommand{\Coker}{\operatorname{Coker}}
\newcommand{\colim}{\operatorname{colim}}

\newcommand{\End}{\operatorname{End}}
\newcommand{\Ext}{\operatorname{Ext}}

\newcommand{\fHom}{\operatorname{\mathcal{H}\!\!\;\mathit{om}}}

\newcommand{\Hom}{\operatorname{Hom}}
\newcommand{\id}{\operatorname{id}}

\newcommand{\Ker}{\operatorname{Ker}}

\newcommand{\Loc}{\operatorname{\mathsf{Loc}}}

\renewcommand{\mod}{\operatorname{\mathsf{mod}}}
\newcommand{\Mod}{\operatorname{\mathsf{Mod}}}
\newcommand{\PHom}{\operatorname{PHom}}

\newcommand{\sHom}{\underline{\Hom}}
\newcommand{\Spec}{\operatorname{Spec}}
\newcommand{\Spc}{\operatorname{Spc}}
\newcommand{\StMod}{\operatorname{\mathsf{StMod}}}
\newcommand{\stmod}{\operatorname{\mathsf{stmod}}}
\newcommand{\Supp}{\operatorname{Supp}}
\newcommand{\supp}{\operatorname{supp}}
\newcommand{\Thick}{\operatorname{\mathsf{Thick}}}

\newcommand{\Tr}{\operatorname{Tr}}
\newcommand{\WInj}{\operatorname{\mathsf{WInj}}}

\newcommand{\winj}{\operatorname{\mathsf{winj}}}

\newcommand{\ac}{\mathrm{ac}} 

\newcommand{\cw}{\mathrm{cw}}
 
\newcommand{\kpac}{k\text{-}\mathrm{pac}}

\newcommand{\pac}{\mathrm{pac}}
\newcommand{\pur}{\mathrm{pur}}

\newcommand{\col}{\colon}

\newcommand{\lto}{\longrightarrow}
\newcommand{\Ra}{\Rightarrow}

\newcommand{\ua}{{\uparrow}}
\newcommand{\wt}{\widetilde}

\newcommand{\xra}{\xrightarrow}

\def\mcV{\mathcal{V}}

\def\sfc{\mathsf c}

\def\sfC{\mathsf C}
\def\sfD{\mathsf D}

\def\sfK{\mathsf K}

\def\sfP{\mathsf P}
\def\sfR{\mathsf R}
\def\sfS{\mathsf S} 
\def\sfT{\mathsf T} 
\def\sfU{\mathsf U}
\def\sfV{\mathsf V}

\def\bbZ{\mathbb Z}

\newcommand{\bsi}{\boldsymbol{i}} 
\newcommand{\bsp}{\boldsymbol{p}}

\newcommand{\bst}{\boldsymbol{t}}

\newcommand{\fp}{\mathfrak{p}}

\newcommand{\eps}{\varepsilon}

\newcommand{\gam}{\varGamma}

\def\Si{\Sigma}

\def\one{\mathbf 1}

\title[Module categories for group algebras]{Module categories for
  group algebras \\ over commutative rings}

\author[Benson, Iyengar, Krause]{Dave Benson, Srikanth B. Iyengar, Henning Krause\\
with an appendix by Greg Stevenson}

\address{Dave Benson \\ 
Institute of Mathematics\\ 
University of Aberdeen\\ 
King's College\\ 
Aberdeen AB24 3UE\\ 
Scotland U.K.}

\address{Srikanth B. Iyengar\\ 
Department of Mathematics\\
University of Nebraska\\ 
Lincoln, NE 68588\\ 
U.S.A.}

\address{Henning Krause\\ 
Fakult\"at f\"ur Mathematik\\ 
Universit\"at Bielefeld\\ 
33501 Bielefeld\\ 
Germany.}

\begin{document}

\begin{abstract} We develop a suitable version of the stable module
category of a finite group $G$ over an arbitrary commutative ring $k$.
The purpose of the construction is to produce a compactly generated
triangulated category whose compact objects are the finitely presented
$kG$-modules.  The main idea is to form a localisation of the usual
version of the stable module category with respect to the filtered
colimits of weakly injective modules. There is also an analogous
version of the homotopy category of weakly injective $kG$-modules and
a recollement relating the stable category, the homotopy category, and
the derived category of $kG$-modules.
\end{abstract}

%\keywords{group algebra, stable module category, weak injective, weak projective}
%\subjclass[2010]{20J06 (primary); 13D45, 16E45, 18E30}

\thanks{Version from 5th August 2012.\\ The authors are grateful
to the Mathematische Forschungsinstitut at Oberwolfach for support
through a ``Research in Pairs'' visit. The research of the second
author was partly supported by NSF grant DMS 1201889.}

\maketitle \setcounter{tocdepth}{1}
\tableofcontents

\section{Introduction}

The purpose of this paper is to develop a version of the stable module
category of a finite group $G$ over an arbitrary commutative ring
$k$. We start with the category $\Mod (kG)$ whose objects are all
$kG$-modules and whose morphisms are the module homomorphisms. If $k$
were a field, the stable module category $\StMod (kG)$ would be formed
by taking the same objects, but with morphisms
\[ 
\sHom_{kG}(M,N)=\Hom_{kG}(M,N)/\PHom_{kG}(M,N)
\] 
where $\PHom_{kG}(M,N)$ is the linear subspace consisting of those
morphisms that factor through some projective $kG$-module. This is
then a triangulated category that is compactly generated, with compact
objects the finitely generated $kG$-modules;
see~\cite{Happel:1988a}. The crucial property that makes this work is
that the group algebra $kG$ is self injective, so projectives modules are the same as
the injective modules.

Over an arbitrary commutative ring $k$ however, projective modules are
no longer injective, so one needs to adjust this
construction. Instead, we use the fact that weakly projective modules
are weakly injective, where ``weakly'' means ``relative to the trivial
subgroup''. Thus we define $\PHom_{kG}(M,N)$ to be the morphisms that
factor through a weakly projective module, and construct $\StMod (kG)$
as above. This is again a triangulated category, and the compact
objects in it are the finitely presented $kG$-modules; we write
$\stmod (kG)$ for the full subcategory of modules isomorphic in
$\StMod (kG)$ to a finitely presented module.

This is not quite good enough to make a decent stable module category,
as there are several problems with this construction:
\begin{enumerate}
\item[--] The category $\StMod (kG)$ is usually not compactly generated by
$\stmod (kG)$.
\item[--] The functor $k_{(|G|)}\otimes_k
-\,\colon\StMod (kG)\to\StMod(k_{(|G|)}G)$ is not an equivalence of
categories, where $k_{(|G|)}$ denotes $k$ localised by inverting all
integers coprime to $|G|$.
\end{enumerate} The reason for both of these is that filtered colimits
of weakly projective modules are not necessarily weakly projective;
see Example~\ref{ex:pz}.

To remedy these deficiencies, we form the Verdier quotient of
$\StMod (kG)$ with respect to the localising subcategory of
\emph{cw-injectives}, namely filtered colimits of weakly injective
modules. The principal features of this quotient category, which we
denote $\StMod^\cw (kG)$, are collected in the result below.

\begin{theorem}
\label{th:StMod-cw} The following statements hold:
\begin{enumerate}
\item The category $\StMod^\cw (kG)$ is a compactly generated triangulated category.
\item The natural functor $\stmod (kG)\to \StMod^\cw (kG)$ identifies
  $\stmod (kG)$ with the full subcategory consisting of the compact
  objects.
\item The functor $k_{(|G|)}\otimes
-\colon\StMod^\cw (kG)\to\StMod^\cw(k_{(|G|)}G)$ is an equivalence of
triangulated categories.
\item An object in $\StMod^\cw (kG)$ is zero if and only if its image
in $\StMod^\cw(k_{(p)}G)$ is zero for each prime $p$ dividing $|G|$.
\end{enumerate}
\end{theorem}

These results are proved in Sections~\ref{se:stable_category} and
\ref{se:base_change}, building on properties of weakly injective
modules and their filtered colimits, presented in the preceding
sections.

The equivalence in (3) above is deduced from general results on
$\StMod^\cw (kG)$ in Section~\ref{se:base_change}, concerning base
change along a homomorphism $k\to k'$ of commutative rings, and
leading to the following local-global principle:

\begin{theorem} A $kG$-module $M$ is cw-injective if and only if
$M_\fp$ is a cw-injective $k_\fp G$-module for every maximal ideal
$\fp\subseteq k$.
\end{theorem}

In Section~\ref{se:htpy_category} we develop a version of the homotopy
category of weakly injective $kG$-modules, namely, we introduce a
Verdier quotient $\sfK_\pur(\WInj kG)$ of the homotopy category of
weakly injective $kG$-modules that has the following properties.

\begin{theorem} 
  The category $\sfK_\pur(\WInj kG)$ is a compactly generated
  triangulated category and the compact objects form a full
  triangulated category that is equivalent to the bounded derived
  category $\sfD^b(\mod kG)$ of finitely presented
  $kG$-modules. Moreover, there is a recollement
  \[ \xymatrix{ \StMod^\cw (kG)\,\ar[rr]&&\,\sfK_\pur(\WInj kG)\,
    \ar[rr]\ar@<1.25ex>[ll]\ar@<-1.25ex>[ll]&& \,\sfD_\pur(\Mod
    kG)\ar@<1.25ex>[ll]\ar@<-1.25ex>[ll] }
\] where $\sfD_\pur(\Mod kG)$ denotes the pure derived category of $kG$-modules.
\end{theorem}

When $k$ is a field, the thick tensor ideals of $\stmod (kG)$, and the
localising tensor ideals of $\StMod (kG)$, have been classified, in
terms of appropriate subsets of the Zariski spectrum of the cohomology
ring $H^*(G,k)$; see
\cite{Benson/Carlson/Rickard:1997a,Benson/Iyengar/Krause:2011b}. One
of our motivations for undertaking the research reported here is to
investigate possible extensions of these results to a general
commutative ring $k$. The tensor product of $kG$-modules again induces
a tensor triangulated structure on the stable module
category. However, in Section~\ref{se:thick} we produce an infinite
chain of thick tensor ideals of $\stmod (\bbZ G)$ for any group $G$
with at least two elements, and explain why this implies that the
classification in terms of cohomological support
\cite{Benson/Carlson/Rickard:1997a} does not carry over verbatim to
this context.

The same example shows that the spectrum of the tensor triangulated
category $\stmod (\bbZ G)$ in the sense of Balmer~\cite{BaSpec} is
disconnected. Again, this is in contrast to the case when the ring of
coefficients is a field. A detailed explanation of this phenomenon can
be found in Appendix~\ref{ap:appendix}.

\section{Weakly projective and weakly injective modules}
\label{se:weak things}

In this section we recall basic results on relative projectives and
relative injectives for group algebras over commutative rings,
obtained by Gasch\"utz~\cite{Gaschuetz:1952a} and
D. G. Higman~\cite{Higman:1954a}, building on the work of Maschke and
Schur. See also \cite[\S3.6]{Benson:1991a}. The novelty, if any, is a
systematic use of Quillen's~\cite{Quillen:1973a} notion of exact
categories.

Throughout this work $G$ is a finite group, $k$ a commutative
ring of coefficients and $kG$ the group algebra. The
category of $kG$-modules is denoted $\Mod (kG)$, and $\mod (kG)$ is its
full subcategory consisting of all finitely presented modules.

\subsection*{Restriction to the trivial subgroup} 

The functor $\Mod (kG)\to \Mod (k)$ that restricts a $kG$-module to the
trivial subgroup has a left adjoint $kG\otimes_k -$ and a right
adjoint $\Hom_k(kG,-)$. These functors are isomorphic, since $G$ is
finite; we identify them, and write $M\ua^G$ for $kG\otimes_k M$; the
$G$ action is given by $h(g\otimes m)=hg \otimes m$. Adjunction
yields, for each $kG$-module $M$, natural morphisms of $kG$-modules
\begin{alignat*}{2}
 &\iota_M\colon M\lto M{\ua^G}&
\quad\text{and}\quad &\pi_M\colon M{\ua^G}\lto M\quad{\text{where}}\\
&\iota_M(m)= \sum_{g\in G}g\otimes g^{-1}m &\quad\text{and}\quad
&\pi_M(g\otimes m)= gm \,.
\end{alignat*} Note that $\iota_M$ is a monomorphism whilst $\pi_M$ is an
epimorphism.

\subsection*{Weakly projective and weakly injective modules}

An \emph{exact category} in the sense of Quillen \cite{Quillen:1973a}
(see also Appendix~A of Keller \cite{Keller:1990a}) is an additive
category with a class of short exact sequences satisfying certain
axioms. For example, every additive category is an exact category with
respect to the \emph{split exact structure}.

We give $\Mod (kG)$ the structure of an exact category with respect to
the $k$-split short exact sequences. Said otherwise, we consider for
$\Mod (k)$ the split exact structure, and declare that a sequence of
$kG$-modules is \emph{$k$-split exact} if it is split exact when
restricted to the trivial subgroup.

For any $kG$-module $M$, the following sequences of $kG$-module
are $k$-split exact:
\begin{gather*} 0\lto M \xra{\iota_M} M\ua^G \lto\Coker\iota_M \lto
0\\ 0\lto \Ker \pi_M \lto M\ua^G \xra{\pi_M} M\lto 0\,.
\end{gather*} Indeed, each element of $M\ua^{G}$ can be written
uniquely as $\sum_{g\in G}g\otimes_{k}m_{g}$. The map assigning such
an element to $m_{1}$, where $1$ is the identity of $G$, is a
$k$-splitting of $\iota_{M}$, whilst the map $M\to M\ua^{G}$ assigning
$m$ to $1\otimes_{k}m$ is a $k$-splitting of $\pi_{M}$.

\begin{definition} A $kG$-module $P$ is said to be \emph{weakly
projective} if $\Hom_{kG}(P,-)$ preserves exactness of $k$-split exact
sequences; equivalently, given a $k$-split exact sequence $0\to L \to
M\to N\to 0$ of $kG$-modules, any morphism $P\to N$ lifts to $M$.

In the same vein, a $kG$-module $P$ is \emph{weakly injective} if
$\Hom_{kG}(-,P)$ preserves exactness of $k$-split exact sequences;
equivalently, given a $k$-split exact sequence $0\to L \to M\to N\to
0$ of $kG$-modules, any morphism $L\to P$ extends to $M$.
\end{definition}

\subsection*{The trace map} 

We recall a characterisation of weakly projective and injective
$kG$-modules in terms of the trace map.

\begin{definition} If $M$ and $N$ are $kG$-modules, the \emph{trace
map}
\[ \Tr_G\colon\Hom_k(M,N)\to\Hom_{kG}(M,N)
\] is defined by
\[ \Tr_G(\theta)(m)=\sum_{g\in G}g(\theta(g^{-1}m)) \quad\text{for
$m\in M$.}
\]
\end{definition}

A straightforward calculation justifies the following properties of
the trace map.

\begin{lemma}
\label{le:tr-ideal} Fix a homomorphism $\theta\in\Hom_k(M,N)$.
\begin{enumerate}
\item If $\alpha\colon M'\to M$ is a morphism of $kG$-modules, then
$\Tr_G(\theta\alpha)=\Tr_G(\theta)\alpha$.
\item If $\beta\colon N \to N'$ is a morphism of $kG$-modules, then
$\Tr_G(\beta\theta)=\beta\Tr_G(\theta)$.\qed
\end{enumerate}
\end{lemma}

For any $kG$-module $M$, let $\theta_M\in\End_k(M{\ua^G})$ be the map
given defined by
\[ \theta_M(g\otimes m)=
\begin{cases} 1\otimes m & g=1 \\ 0 & g\ne 1.
\end{cases}
  \]

\begin{lemma}
\label{le:tr-ind} The map $\Tr_G(\theta_M)$ is the identity morphism on
$M{\ua^G}$.\qed
\end{lemma}

\begin{theorem}
\label{th:trace} Let $\rho\colon M \to N$ be a morphism of
$kG$-modules. The following conditions are equivalent:
\begin{enumerate}
\item In $\Mod (kG)$, the map $\rho$ factors through the surjection
$\pi_N\colon N{\ua^G}\to N$.
\item In $\Mod (kG)$, the map $\rho$ factors through the injection
$\iota_M\colon M \to M{\ua^G}$.
\item There exists $\theta\in\Hom_k(M,N)$ such that
$\Tr_G(\theta)=\rho$.
\end{enumerate}
\end{theorem}
\begin{proof} (1) $\Ra$ (3): If there exists a homomorphism
$\rho'\colon M\to N{\ua^G}$ of $kG$-modules with
$\rho=\pi_N\rho'$ then
$\rho=\Tr_G(\pi_N\theta_N\rho')$, by Lemmas
\ref{le:tr-ideal} and \ref{le:tr-ind}.

(2) $\Ra$ (3) follows as above.

(3) $\Ra$ (1) and (2): Suppose there exists $\theta\in\Hom_k(M,N)$
with $\Tr_G(\theta)=\rho$. The map $\rho'\in\Hom_{kG}(M,N{\ua^G})$
with $\rho'(m)=\sum_{g\in G} g\otimes \theta(g^{-1}m)$ satisfies
$\pi_N\rho'=\rho$, and $\rho''\in\Hom_{kG}(M{\ua^G},N)$ with
$\rho''(g\otimes m)=g\theta(m)$ satisfies $\rho''\iota_M=\rho$.
\end{proof}

\subsection*{Criteria for weak projectivity and weak injectivity} The
result below is due to D.~G.~Higman \cite{Higman:1954a}. Part (5) is
called \emph{Higman's criterion} for weak projectivity.

\begin{theorem}
\label{th:wp=wi} The following conditions on a $kG$-module $P$ are
equivalent:
\begin{enumerate}
\item The module $P$ is weakly projective.
\item The natural morphism $\pi_P\colon P{\ua^G}\to P$ is a split
surjection.
\item The module $P$ is weakly injective.
\item The natural morphism $\iota_P\colon P\to P{\ua^G}$ is a split
injection.
\item There exists $\theta\in\End_{k}(P)$ such that
$\Tr_{G}(\theta)=\id_{P}$.
\item There exists a $kG$-module $M$ such that $P$ is a direct summand
of $M{\ua^G}$.
\end{enumerate} In particular, $M{\ua^{G}}$ is weakly projective and
weakly injective for any $kG$-module $M$.
\end{theorem}

\begin{proof} We first prove that conditions (1), (2), and (5) are
equivalent; analogous arguments settle the equivalence of (3), (4) and
(5).

(1) $\Ra$ (2): If $P$ is weakly projective then the identity morphism
of $P$ lifts to a splitting of $P{\ua^G}\to P$.

(2) $\Ra$ (5): If $\alpha\colon P\to P{\ua^G}$ is a splitting then by
Lemma \ref{le:tr-ideal} we have
\[
\Tr_G(\pi_P\theta_{P{\ua^G}}\alpha)=\pi_P\Tr_G(\theta_{P{\ua^G}})\alpha
=\pi_P\alpha=\id_P\,.
\]

(5) $\Ra$ (1): Let $\theta\in\End_k(P)$ with
$\Tr_G(\theta)=\id_P$. Given a $k$-split exact sequence $0\to L\to
M\to N\to 0$ and a morphism $\alpha\col P\to N$ of $kG$-modules,
choose a $k$-lift $\wt\alpha\colon P\to M$ of $\alpha$; it follows
from Lemma \ref{le:tr-ideal} that $\Tr_G(\wt\alpha\theta)$ is a
$kG$-lift of $\alpha$.

(4) $\Ra$ (6): Take $M=P$.

(6) $\Ra$ (5): A module of the form $M{\ua^G}$ satisfies (5) by Lemma
\ref{le:tr-ind}. It remains to observe that this property descends to
direct summands, by Lemma~\ref{le:tr-ideal}.
\end{proof}

The theorem above shows that the exact category $\Mod (kG)$ is a
\emph{Frobenius category}. This means that there are enough projective
objects and enough injective objects, and that both coincide; see
Section~I.2 of Happel \cite{Happel:1988a} for details.

When $k$ is a field, weak injectivity is equivalent to
injectivity. The following corollary of Theorem~\ref{th:wp=wi} is thus
a generalisation of Maschke's theorem.

\begin{corollary}
\label{cor:maschke} 
Let $k$ be a commutative ring and $G$ a finite group. The following
conditions are equivalent:
\begin{enumerate}
\item Each $kG$-module is weakly injective.
\item The trivial $kG$-module $k$ is weakly injective.
\item The order of $G$ is invertible in $k$.
\end{enumerate}
\end{corollary}

\begin{proof} Both non-trivial implications follow from Higman's
criterion: (2) $\Ra$ (3) because $\Tr_{G}(\theta)=|G|\theta$ for any
$\theta\in\End_{k}(k)\cong k$, whilst (3) $\Ra$ (1) because
$\Tr_{G}(\frac 1{|G|}\id_{P})=\id_{P}$ for any $kG$-module $P$.
\end{proof}

\section{An example}

In this section, we construct modules that are filtered colimits of
weakly injective $kG$-modules but are not themselves weakly injective.
In what follows, given an integer $n$, we write $k_{(n)}$ for
the localisation of $k$ that inverts all integers $r$ such that
$(r,n)=1$. For any $kG$-module $M$, we denote $M_{(n)}$ the
$kG$-module $k_{(n)}\otimes_{k}M$; this can also be described as a
filtered colimit:
\[ M_{(n)}=\colim_{(r,n)=1} (M\xra{r} M)\,.
\]

\begin{definition}
\label{defn:pm} Let $k$ be a commutative ring, $G$ a finite group, and
$n$ an integer. For any $kG$-module $M$ define $\varGamma_{n}{M}$ by
forming the following exact sequence
\begin{equation}\label{eq:iota}
0\lto M\lto M\ua^G\oplus M_{(n)}\lto\varGamma_{n}{M}\lto 0
\end{equation} 
where the monomorphism sends $m\in M$ to $(\iota_M(m), m)$. This
sequence is $k$-split exact, since $\iota_{M}$ is $k$-split.
\end{definition}

\begin{theorem}
\label{th:Y} If $|G|$ divides $n$, the $kG$-module $\varGamma_{n}{M}$
is a filtered colimit of weakly injective $kG$-modules.
\end{theorem}

\begin{proof} It is expedient to use the stable category,
$\StMod (kG)$, defined in Section~\ref{se:stable_category}.

The $kG$-module $\varGamma_{n}{M}$ is given by the following pushout
\[ \xymatrix{ M\ar[r]^-{\iota_M}\ar[d]&M\ua^G \ar[d] \\ M_{(n)} \ar[r]
& \varGamma_nM}
  \] and this is a filtered colimit of pushouts of the form
\[ \xymatrix{ M\ar[r]^-{\iota_M}\ar[d]^r&M\ua^G\ar[d] \\ M\ar[r] &
M_{(r)}}
  \] 
with $(r,n)=1$. In particular, $\varGamma_{n}{M}$ is a filtered
colimit of the $M_{(r)}$. The hypothesis is that $|G|$ divides $n$, so
one gets
\[ 
\Tr_{G}(\frac n{|G|}\id_{M})=\frac n{|G|}\Tr_{G}(\id_{M}) = \frac
n{|G|} (|G| \id_{M}) = n\id_{M}\,.
\] 
Therefore the multiplication morphism $M\xra{n} M$ factors through
a weakly injective module, by Theorem~\ref{th:trace}. For each integer
$r$ with $(r,n)=1$, there exist integers $a$ and $b$ such that
$na+rb=1$, so the preceding computation implies that the morphism
$M\xra{r} M$ is an isomorphism in the stable category $\StMod (kG)$,
with inverse the multiplication by $a$. It follows that the morphism
on the right in the diagram is also an isomorphism. Therefore each $M_{(r)}$
is weakly injective.
\end{proof}

The $kG$-module $\varGamma_{n}{M}$ need not be itself weakly
injective.

\begin{example}\label{ex:pz} 
  Let $G$ be a finite group and set $k=\bbZ$.  Given $n>1$, the $\bbZ
  G$-module $\varGamma_{n}{\bbZ}$ constructed above is weakly
  injective if and only if $|G|=1$.

  To prove this, first observe that $\Hom_{\bbZ}(\bbZ_{(n)},\bbZ)=0$
  so $\Hom_{\bbZ G}(\bbZ_{(n)},\bbZ)=0$ as well. If
  $\varGamma_{n}{\bbZ}$ is weakly injective, then the sequence
  \eqref{eq:iota} splits and $\iota_\bbZ$ is a split monomorphism,
  since the second component of the inverse of $\bbZ\to
  \bbZ\ua^G\oplus \bbZ_{(n)}$ is zero.  It follows that $\bbZ$ is
  weakly injective, by Theorem~\ref{th:wp=wi}, and therefore $|G|$ is
  invertible in $\bbZ$, by Corollary~\ref{cor:maschke}. Thus
  $|G|=1$. The other direction is clear.
\end{example}

\section{Filtered colimits of weakly injective modules}
\label{se:colimits} Motivated by the example in the previous section
we study filtered colimits of weakly injective modules.

\begin{definition} We say that a $kG$-module is \emph{cw-injective} if
it is isomorphic in $\Mod (kG)$ to a filtered {\large\bf c}olimit of
{\large\bf w}eakly injective modules.
\end{definition}

Weakly injective modules are cw-injective, but not conversely; see
Example \ref{ex:pz}. However, these notions coincide for finitely
presented modules, as is proved in part (2) of the next result.

\begin{lemma}
\label{le:cw-perp} Let $M$ be a $kG$-module.
\begin{enumerate}
\item The module $M$ is cw-injective if and only if every morphism
  $N\to M$ with $N$ finitely presented factors through a weakly
  injective module.
\item If $M$ is finitely presented and cw-injective, then it is weakly
injective.
\end{enumerate}
\end{lemma}

\begin{proof} (1) If $M=\colim M_\alpha$, a filtered colimit of weakly
injective modules $M_\alpha$, then each morphism $N\to M$ with $N$
finitely presented factors through one of the structure morphisms
$M_\alpha\to M$.

As to the converse, one has that $M=\colim_{N\to M} N$, with the
colimit taken over the category of morphisms $N\to M$ where $N$ runs
through all finitely presented $kG$-modules. Suppose that each
morphism $N\to M$ factors through a weakly injective module. Thus each
morphism $N\to M$ factors through $N\ua^G$, by
Theorem~\ref{th:trace}. The module $N\ua^G$ is finitely presented and
it follows that the morphisms $N\to M$ with $N$ weakly injective form
a cofinal subcategory. Thus $M$ is a filtered colimit of weakly
injective $kG$-modules.

(2) Let $M$ be a finitely presented $kG$-module. If $M$ is
cw-injective, then (1) implies that the identity morphism $\id_M$
factors through a weakly injective module. Thus $M$ is weakly
injective, by Theorem~\ref{th:wp=wi}.
\end{proof}

Next we present criteria for cw-injectivity in terms of the purity of
certain canonical morphisms. For background on purity, see the books
of Jensen and Lenzing \cite{Jensen/Lenzing:1989a} and Prest
\cite{Prest:1988a,Prest:2009a}.

\begin{definition}
\label{defn:purity} An exact sequence
\[ 0\to M' \to M \to M'' \to 0
\] of $kG$-modules is said to be \emph{pure exact} if for every right
$kG$-module $N$, the sequence
\[ 0 \to N\otimes_{kG} M' \to N \otimes_{kG} M \to N \otimes_{kG}
M''\to 0
\] is exact. This is equivalent to the statement that for any finitely
presented $kG$-module $P$, each morphism $P \to M''$ lifts to a
morphism $P\to M$, and to the statement that the sequence is a
filtered colimit of split exact sequences.

A morphism $M\to M''$ appearing in a pure exact sequence as above is
called a \emph{pure epimorphism}, whilst $M'\to M$ is called a
\emph{pure monomorphism}.
\end{definition}

\iffalse A $kG$-module $M$ is said to be \emph{pure projective} if for
every pure exact sequence \eqref{eq:pes}, every morphism $M\to C$
lifts to a morphism $M\to B$. Thus finitely presented modules are pure
projective.

A $kG$-module $M$ is said to be \emph{pure injective} if for every
pure exact sequence \eqref{eq:pes}, every morphism $A\to M$ extends to
a morphism $B\to M$.  \fi

\subsection*{Criteria for cw-injectivity} The following tests for
cw-injectivity are analogues of the criterion for weak injectivity
formulated in Theorem~\ref{th:wp=wi}.

\begin{theorem}\label{th:cw-injective} 
  The following properties of a $kG$-module $M$ are equivalent:
\begin{enumerate}
\item The module $M$ is cw-injective.
\item The natural morphism $\pi_{M}\colon M{\ua^G} \to M$ is a pure epimorphism.
\item The natural morphism $\iota_{M}\colon M \to M{\ua^G}$ is a pure monomorphism.
\end{enumerate}
\end{theorem}
\begin{proof} (1) $\Ra$ (2) and (3): If $M$ is a filtered colimit of
weakly injective modules $M_\alpha$, then $M{\ua^G}\to M$ is a
filtered colimit of split epimorphisms $M_\alpha{\ua^G}\to M_\alpha$
and hence is a pure epimorphism. Similarly, $M\to M{\ua^G}$ is a filtered
colimit of split monomorphisms $M_\alpha\to M_\alpha{\ua^G}$ and hence
is a pure monomorphism.

(2) $\Ra$ (1): When $M{\ua^G}\to M$ is a pure epimorphism any morphism
$N\to M$ with $N$ finitely presented lifts to $M{\ua^G}$, so
Lemma~\ref{le:cw-perp} yields that $M$ is cw-injective.

(3) $\Ra$ (1): Let $N\to M$ be a morphism of $kG$-modules with $N$
finitely presented. We obtain the following commutative diagram with
exact rows.
\[ \xymatrix{
0\ar[r]&N\ar[r]\ar[d]&N\ua^G\ar[r]\ar[d]&N'\ar[r]\ar[d]&0\\
0\ar[r]&M\ar[r]&M\ua^G\ar[r]&M'\ar[r]&0}
\] When $M\to M{\ua^G}$ is a pure monomorphism, $M\ua^G\to M'$ is a
pure epimorphism, so the morphism $N'\to M'$ lifts to $M\ua^G$; note
that $N'$ is also finitely presented. Thus the morphism $N\to M$
extends to $N\ua^G$, so $M$ is cw-injective, by
Lemma~\ref{le:cw-perp}.
\end{proof}

\begin{corollary}
\label{co:cw-injective} The class of cw-injective modules is closed
under
\begin{enumerate}
\item filtered colimits,
\item products, and
\item pure submodules and pure quotient modules.
\end{enumerate}
\end{corollary}

\begin{proof} We use the characterisation from Lemma~\ref{le:cw-perp}
to verify (1) and (2).

(1) Let $M=\colim M_\alpha$ be a filtered colimit of cw-injective
modules. Each morphism $N\to M$ with $N$ finitely presented factors
through one of the structure morphisms $M_\alpha\to M$. The morphism
$N\to M_\alpha$ then factors through a weakly injective module. Thus
$M$ is cw-injective.

(2) Let $M=\prod_\alpha M_\alpha$ be a product of cw-injective modules
and fix a finitely presented module $N$. Each morphism $N\to M_\alpha$
factors through a weakly injective module and therefore through the
natural morphism $N\to N\ua^G$. Thus each morphism $N\to M$ factors
through $N\ua^G$, and it follows that $M$ is cw-injective.

(3) Let $M\to N$ be a pure monomorphism with $N$ cw-injective.
An application of Theorem~\ref{th:cw-injective} yields that
the composite $M\to N\to N\ua^G$ is a pure monomorphism; it factors
through the natural monomorphism $M\to M\ua^G$, so the latter is pure
as well. Another application of Theorem~\ref{th:cw-injective} shows
that $M$ is cw-injective.

The argument for a pure quotient module is analogous
\end{proof}

\section{Stable module categories}
\label{se:stable_category}

In this section we introduce the stable module category $\StMod (kG)$
and an appropriate variant which we denote by $\StMod^\cw (kG)$.  We
begin with a brief discussion of compactly generated triangulated
categories.

\subsection*{Compactly generated triangulated categories}

Let $\sfT$ be a triangulated category with suspension
$\Si\colon\sfT\xra{\sim}\sfT$. For a class $\sfC$ of objects we define
its \emph{perpendicular categories} to be the full subcategories
\begin{align*} \sfC^\perp&=\{Y\in\sfT\mid\Hom_\sfT(\Si^n X,Y)=0\text{
for all }X\in\sfC,n\in\mathbb Z\}\\
^\perp\sfC&=\{X\in\sfT\mid\Hom_\sfT(X,\Si^n Y)=0\text{ for all
}Y\in\sfC,n\in\mathbb Z\}.
\end{align*} We write $\Thick(\sfC)$ and $\Loc(\sfC)$ for the thick
subcategory and the localising subcategory, respectively, of $\sfT$
generated by $\sfC$.

Suppose that $\sfT$ has set-indexed coproducts. An object $X\in\sfT$
is \emph{compact} if the representable functor $\Hom_\sfT(X,-)$
preserves set-indexed coproducts. The category $\sfT$ is
\emph{compactly generated} if there is a set $\sfC$ of compact objects
such that $\sfT$ admits no proper localising subcategory containing
$\sfC$.

The following proposition provides a useful method to construct
compactly generated triangulated categories and to identify its
subcategory of compact objects. As usual, $\sfT^{\sfc}$ denotes the
full subcategory of compact objects in $\sfT$.

\begin{proposition}
\label{pr:compact} Let $\sfT$ be a triangulated category with
set-indexed coproducts, $\sfC$ a set of compact objects in $\sfT$, and
set $\sfV=\sfC^\perp$.  Then the composite
  \[ \Loc(\sfC) \rightarrowtail\sfT\twoheadrightarrow\sfT/\sfV
\] is an equivalence of triangulated categories.  Moreover, the
category $\sfT/\sfV$ is compactly generated and the equivalence
identifies $\Thick(\sfC)$ with $(\sfT/\sfV)^{\sfc}$.
\end{proposition}

\begin{proof} 
  Set $\sfU=\Loc(\sfC)$. Observe that since $\sfC$
  consists of compact objects, the Verdier quotient $\sfT/\sfU$ does
  exist, in that the morphisms between a given pair of objects form a
  set; see \cite[Proposition~9.1.19]{Neeman:2001a}. It follows that
  the Verdier quotient $\sfT/\sfV$ exists, since $\sfU^\perp=\sfV$. Moreover,
  the composite $\sfU\rightarrowtail\sfT\twoheadrightarrow\sfT/\sfV$
  is an equivalence, by \cite[Theorem~9.1.16]{Neeman:2001a}, with the
  warning that there $\sfU^\perp$ is denoted $^\perp\sfU$.  The
  category $\sfU$ is compactly generated by construction.  The
  description of the compacts follows from standard properties of
  compact objects; see Lemma~2.2 in \cite{Neeman:1992b}.
\end{proof}

\begin{corollary}
\label{cor:compact} If $\sfT'\subseteq\sfT$ is a localising
subcategory containing $\Loc(\sfC)$, then the natural functor
\[ \sfT'/(\sfV\cap\sfT')\lto\sfT/\sfV
\] is an equivalence of triangulated categories.
\end{corollary}

\begin{proof} This follows from Lemma~\ref{le:equivalence} below,
since the functor induces an equivalence between the categories of
compact objects.
\end{proof}

The following test for equivalence of triangulated categories is
implicit in \cite[\S4.2]{Keller:1994a}; for details see \cite[Lemma
4.5]{Benson/Iyengar/Krause:2011b}.

\begin{lemma}
\label{le:equivalence} Let $F\col\sfS \to \sfT$ be an exact functor
between compactly generated triangulated categories and suppose $F$
preserves set-indexed coproducts. If $F$ restricts to an equivalence
$\sfS^c\xrightarrow{\sim} \sfT^c$, then $F$ is an equivalence of
categories. \qed
\end{lemma}

\subsection*{The stable module category}

The following lemma describes the morphisms that are annihilated when
one passes to the stable module category; it is an immediate
consequence of Theorems~\ref{th:trace} and \ref{th:wp=wi}.

\begin{lemma}
\label{le:stable0} For a morphism $M \to N$ of
$kG$-modules, the following conditions are equivalent:
\begin{enumerate}
\item The morphism factors through some weakly projective module.
\item The morphism factors through the natural epimorphism
  $N{\ua^G}\to N$.
\item The morphism factors through some weakly injective module.
\item The morphism factors through the natural monomorphism $M \to
  M{\ua^G}$.\qed
\end{enumerate}
\end{lemma}

The \emph{stable module category} $\StMod (kG)$ has the same objects as
$\Mod (kG)$, but the morphisms are given by
\[ \sHom_{kG}(M,N)=\Hom_{kG}(M,N)/\PHom_{kG}(M,N)
\] where $\PHom_{kG}(M,N)$ is the $k$-submodule of $\Hom_{kG}(M,N)$
consisting of morphisms $M \to N$ satisfying the equivalent conditions
of Lemma~\ref{le:stable0}.  As $\Mod (kG)$ is a Frobenius category,
$\StMod (kG)$ is a triangulated category; see
\cite[Section~I.2]{Happel:1988a}. The exact triangles in it are
induced by the $k$-split exact sequences of $kG$-modules.

Let $\stmod (kG)$ be the full subcategory of $\StMod (kG)$ consisting of
objects isomorphic to finitely presented $kG$-modules. This is a thick
subcategory consisting of compact objects.  The result below is an
immediate consequence of Lemma \ref{le:cw-perp}.

\begin{lemma}
\label{le:perp} In $\StMod (kG)$, the perpendicular category $(\stmod
kG)^\perp$ coincides with the localising subcategory whose objects are
the cw-injective $kG$-modules.\qed
\end{lemma}

Set $\sfV=(\stmod kG)^\perp$ and modify the stable module category by
defining
\[\StMod^\cw (kG)=\StMod (kG)/\sfV.\] This Verdier quotient exists by
Proposition~\ref{pr:compact}, since $\stmod (kG)$ consists of compact
objects.  Given Lemma~\ref{le:perp}, the following result is a special
case of Proposition~\ref{pr:compact}.

\begin{theorem}
\label{th:stmodcw} The composite functor
\[ \Loc(\stmod kG) \rightarrowtail\StMod (kG) \twoheadrightarrow
\StMod^\cw (kG)
\] is an equivalence of triangulated categories. The category
$\StMod^\cw (kG)$ is compactly generated and the equivalence identifies
$\stmod (kG)$ with the full subcategory consisting of all compacts
objects of $\StMod^\cw (kG)$. \qed
\end{theorem}

For later use, we record a remark about functors between exact
categories.

\begin{lemma}
\label{le:exact-functor} Let $F\colon \sfC\to \sfD$ be a functor
between exact categories. If $F$ admits a right adjoint that is exact,
then $F$ preserves projectivity; if it admits a left adjoint that is
exact, then $F$ preserves injectivity.
\end{lemma}

\begin{proof} Suppose $F$ has a right adjoint, say $F'$, that is
exact. For any projective object $X$ of $\sfC$, it follows from the
isomorphism
\[ \Hom_\sfD(F(X),-)\cong\Hom_\sfC(X,F'(-))
\] that $\Hom_{\sfD}(F(X),-)$ maps an epimorphism in $\sfD$ to a
surjective map, that is to say, that $F(X)$ is a projective object of
$\sfD$.

The argument for the other case is exactly analogous.
\end{proof}

\subsection*{The tensor product}

The category of $kG$-modules carries a tensor product; it is the
tensor product over $k$ with diagonal $G$-action. Note that this
tensor product is exact in each variable with respect to the $k$-split
exact structure.

\begin{lemma} 
\label{lem:tensor}
Let $M,N$ be $kG$-modules.
\begin{enumerate}
\item If $M$ is weakly injective, then $M\otimes_kN$ is weakly
injective.
\item If $M$ is cw-injective, then $M\otimes_kN$ is cw-injective.
\end{enumerate}
\end{lemma}

\begin{proof} The functor $-\otimes_kN$ has $\Hom_k(N,-)$ as a right
adjoint that is exact with respect to the $k$-split exact
structure. Thus, Lemma~\ref{le:exact-functor} implies that
$-\otimes_kN$ preserves weak projectivity, which coincides with weak
injectivity. The functor $-\otimes_kN$ preserves filtered colimits,
and therefore cw-injectivity as well.
\end{proof}

It follows from the lemma that the tensor product on $\Mod (kG)$ passes
to a tensor product on $\StMod (kG)$; this is compatible with the
triangulated structure, so that $\stmod (kG)$ and $\StMod (kG)$ are
tensor triangulated categories. The cw-injective $kG$-modules form a
tensor ideal in $\StMod (kG)$ and hence $\StMod^\cw (kG)$ inherits a
tensor triangulated structure.

\section{Base change}
\label{se:base_change} In this section we prove a number of results
that track the behaviour of the stable module category
$\StMod^\cw (kG)$ under changes of the coefficient ring $k$. We remind
the reader that $\Mod (kG)$ is the category of $kG$-modules with exact
structure defined by the $k$-split short exact sequences.

Let $\phi\colon k\to k'$ be a homomorphism of commutative
rings. Assigning the element $\sum \alpha_g g$ to $\sum \phi(\alpha_g)
g$ is then a homomorphism of rings $kG\to k'G$.  We consider the
corresponding \emph{restriction functor} $\Mod (k'G)\to\Mod (kG)$ and the
\emph{base change functor} $k'\otimes_k-\colon\Mod (kG)\to\Mod (k'G)$.

\subsection*{Exact structure} Base change and restriction are
compatible with the exact structure for modules over $kG$ and $k'G$
respectively.

\begin{lemma}
\label{le:base-change-exact} Let $\phi\colon k\to k'$ be a ring
homomorphism.
\begin{enumerate}
\item Restriction takes $k'$-split exact sequences of $k'G$-modules to
$k$-split exact sequences of $kG$-modules.
\item Base change takes $k$-split exact sequences of $kG$-modules to
$k'$-split exact sequences of $k'G$-modules.
\item Restriction and base change preserve weak injectivity and
cw-injectivity.
\end{enumerate}
\end{lemma}
\begin{proof} Parts (1) and (2) are clear. As to (3), it follows from
Lemma~\ref{le:exact-functor} that restriction and base change
preserve weak injectivity, because the former is right adjoint to the
latter, and they are both exact functors; one uses again the property
that weak injectivity and weak projectivity coincide. Since both
functors also preserve filtered colimits, they preserve cw-injectivity
as well.
  \end{proof}

\subsection*{Localisation at $|G|$} Recall that for any $kG$-module
$M$ and  integer $n$ we write $M_{(n)}$ for the localisation
of $M$ inverting all integers coprime to $n$.

\begin{theorem}
\label{th:local} For any integer $n$ that is divisible by $|G|$ the
localisation functor
\[ (-)_{(n)} \colon \StMod^\cw (kG) \lto\StMod^\cw(k_{(n)}G)
\] is an equivalence of triangulated categories, and induces an
equivalence
\[ (-)_{(n)} \colon \stmod (kG) \lto \stmod(k_{(n)}G).
\]
\end{theorem}

\begin{proof} To begin with note that, for any $n$, localisation does
induce an exact functor of triangulated categories, by
Lemma~\ref{le:base-change-exact}. It suffices to prove that this is an
equivalence on $\StMod^\cw (kG)$; the second equivalence follows by
restriction to compact objects; see Theorem~\ref{th:stmodcw}.

Consider the restriction functor
\[ 
F\colon \StMod^\cw(k_{(n)}G) \lto \StMod^\cw (kG).
\] 
The composite $F(-)_{(n)}$ is the identity.  As to the other
composite, for each $kG$-module $M$, the morphism in
Definition~\ref{defn:pm} induces an exact triangle
\[ 
M \lto F(M_{(n)}) \lto \varGamma_{n}{M}\lto
\] 
in $\StMod (kG)$, as $M{\ua^{G}}$ is zero there. By
Theorem~\ref{th:Y}, the $kG$-module $\varGamma_{n}{M}$ is
cw-injective, so the first morphism in the triangle defines a natural
isomorphism in $\StMod^\cw (kG)$ from the identity functor to
$F((-)_{(n)})$. Thus $(-)_{(n)}$ is an equivalence of categories with
inverse $F$.
\end{proof}

\begin{remark} As Example~\ref{ex:pz} shows, $\varGamma_{n}{M}$ need
not be weakly injective, so the localisation functor
\[ (-)_{(n)} \colon \StMod (kG) \lto\StMod(k_{(n)}G)
\] is not necessarily an equivalence of categories.
\end{remark}

Concerning Lemma~\ref{le:base-change-exact}, it is easy to construct
examples of $kG$-modules that are weakly injective after base change
along a homomorphism $k\to k'$, but are not themselves weakly
injective. We now focus on certain classes of homomorphisms where such
a phenomenon cannot occur.

\subsection*{Pure monomorphisms} The notion of a pure monomorphism has
been recalled in Definition~\ref{defn:purity}.

\begin{lemma}
\label{le:pure-mono} Let $\phi\colon k\to k'$ be a homomorphism of
rings that is a pure monomorphism of $k$-modules, and let $M$ be a
$kG$-module.  If $k'\otimes_k M$ is cw-injective as a $kG$-module,
then so is $M$.
\end{lemma}

\begin{proof} The canonical morphism $\phi\otimes_{k} M\col M\to
  k'\otimes_k M$ is a pure monomorphism of $kG$-modules, since for any
  right $kG$-module $N$ we have
\[N\otimes_{kG}(\phi\otimes_{k}M)\cong\phi\otimes_{k}(N\otimes_{kG}M)\]
where the second is a monomorphism, by the hypothesis on $\phi$. Thus
the desired statement follows from the fact that a pure submodule of a
cw-injective module is cw-injective, by
Corollary~\ref{co:cw-injective}.
\end{proof}

\begin{proposition}
  \label{pr:pure-cover} Let $\{k\to k_{i}\}_{i\in I}$ be homomorphisms
  of rings such that the induced map $X\to
  \prod_{i}(k_{i}\otimes_{k}X)$ is a monomorphism for any $k$-module
  $X$.  If $M$ is a $kG$-module with the property that
  $k_{i}\otimes_{k}M$ is cw-injective as $kG$-module for each $i\in
  I$, then $M$ is cw-injective.
\end{proposition}

\begin{proof} 
We claim that for any $kG$-module $M$, the natural morphism
\[ \eta_{M}\colon M \lto \prod_{i\in I}(k_{i}\otimes_{k}M)
\] is a pure monomorphism of $kG$-modules. Indeed, for any right
$kG$-module $N$ the morphism $\eta_{(N\otimes_{kG}M)}$ factors as
\[ N \otimes_{kG} M \xra{\ N\otimes_{kG}\eta_{M}\ }
N\otimes_{kG}\prod_{i\in I}(k_{i}\otimes_{k}M)\lto \prod_{i\in I}\big(
k_{i}\otimes_{k}(N\otimes_{kG}M)\big )
\] where the morphism on the right is the canonical one. Since
$\eta_{(N\otimes_{kG}M)}$ is a monomorphism, by hypothesis, so is
$N\otimes_{kG}\eta_M$.  This settles the claim.

Now if each $k_{i}\otimes_{k}M$ is cw-injective as a $kG$-module, then
so is their product and hence also its pure submodule $M$, by
Corollary~\ref{co:cw-injective}.
\end{proof}

\subsection*{A local-global principle} 
From Proposition~\ref{pr:pure-cover} we obtain the following
local-to-global type results for detecting cw-injectivity.  As usual,
the localisation of a $k$-module $M$ at a prime ideal $\fp$ of $k$ is
denoted $M_\fp$.

\begin{corollary}
\label{co:local-global} A $kG$-module $M$ is cw-injective if and only
if $M_\fp$ is a cw-injective $k_\fp G$-module for every maximal ideal
$\fp\subseteq k$.
\end{corollary}

\begin{proof} If $M$ is cw-injective, so is each $M_\fp$ by
Lemma~\ref{le:base-change-exact}. The converse is a special case of
Proposition~\ref{pr:pure-cover}, since the canonical map $X\to
\prod_\fp X_\fp$ is a monomorphism for each $k$-module $X$; see
\cite[Theorem~4.6]{Matsumura:1986a}.
\end{proof}
  
The next result adds to Theorem~\ref{th:local}.

\begin{corollary} 
  Fix an integer $n$ divisible by $|G|$, and let $M$ be a
  $kG$-module. The following conditions are equivalent:
\begin{enumerate}
\item The $kG$-module $M$ is cw-injective.
\item The $k_{(n)}G$-module $M_{(n)}$ is cw-injective.
\item The $k_{(p)}G$-module $M_{(p)}$ is cw-injective for every prime
$p$ dividing $n$.
\end{enumerate}
\end{corollary}
\begin{proof} (1) $\Leftrightarrow$ (2): This follows from
Theorem~\ref{th:local}.

(2) $\Ra$ (3): This is by Lemma~\ref{le:base-change-exact}, applied to
the homomorphism $k_{(n)}\to k_{(p)}$.

(3) $\Ra$ (2): We apply Proposition~\ref{pr:pure-cover} to the family
$k_{(n)}\to k_{(p)}$, as $p$ ranges over the primes dividing $n$.  To this
end, let $X$ be a $k_{(n)}$-module, and note that there is a natural
isomorphism
\[ k_{(p)}\otimes_{k_{(n)}}X \xra{\sim} \bbZ_{(p)}\otimes_{\bbZ_{(n)}}X
\] with $X$ viewed as a $\bbZ_{(n)}$-module via restriction. Since the
ideals $(p)$ in $\bbZ_{(n)}$, where $p$ divides $n$, is a complete
list of maximal ideals of $\bbZ_{(n)}$, it follows from
\cite[Theorem~4.6]{Matsumura:1986a} that the induced map $X\to
\prod_{p}X_{(p)}$ is a monomorphism, as desired.
\end{proof}

\subsection*{Epimorphisms}

Recall that a homomorphism $R\to S$ of rings is an \emph{epimorphism}
(in the category of rings) if and only if the canonical morphism
$S\otimes_R M\to M$ is an isomorphism for each $S$-module $M$;
equivalently, if the multiplication map $S\otimes_{R}S\to S$ is an
isomorphism. This means that the restriction functor $\Mod (S)\to \Mod
(R)$ is fully faithful. For example, each localisation $R\to
R[\Si^{-1}]$ inverting the elements of a subset $\Si\subseteq R$ is an
epimorphism.

\begin{lemma} If $\phi\colon k\to k'$ is a ring epimorphism, then so
is the induced homomorphism $kG\to k'G$.
\end{lemma}

\begin{proof} If $\phi$ is an epimorphism, then the induced morphism
$k'\otimes_k k'\to k'$ is an isomorphism. It follows that the induced
morphism $k'G\otimes_{kG}k'G\to k'G$ is an isomorphism. Thus $kG\to
k'G$ is a ring epimorphism.
\end{proof}

\begin{proposition} Let $k\to k'$ be a ring epimorphism and set
\begin{align*} \sfU&=\{M\in\StMod^\cw (kG)\mid k'\otimes_{k} M=0\}\\
\sfV&=\{M\in\StMod^\cw (kG)\mid M\xra{\sim}k'\otimes_{k}M\,.  \}
\end{align*} Then base change induces an equivalence of triangulated
categories
\[ \StMod^\cw (kG)/\sfU\stackrel{\sim}\lto\StMod^\cw (k'G)
\] while restriction induces an equivalence of triangulated categories
\[ \StMod^\cw (k'G)\stackrel{\sim}\lto \sfV.
\]
\end{proposition}

\begin{proof} Restriction followed by base change is naturally
isomorphic to the identity. By general principles, this implies that
restriction is fully faithful, while base change is up to an
equivalence a quotient functor; see Proposition~1.3 of Chap.~I in
\cite{Gabriel/Zisman:1967a}.
\end{proof}

\begin{remark}
  \label{rem:epimorphisms} Let $k\to k'$ be an epimorphism such that
  $k'$ is finitely presented as a $k$-module, and set
\begin{align*} \sfC&=\{M\in\stmod (kG)\mid k'\otimes_{k} M=0\}\\
\sfD&=\{M\in\stmod (kG)\mid M\xra{\sim}k'\otimes_{k} M\}.
\end{align*} Then $\sfC$ and $\sfD$ form thick tensor ideals of
$\stmod (kG)$ such that the inclusion of $\sfC$ admits a right adjoint
while the inclusion of $\sfD$ admits a left adjoint. Moreover,
$\sfC^\perp=\sfD$ and $\sfC={^\perp\sfD}$. Thus the inclusions of
$\sfC$ and $\sfD$ and their adjoints form a localisation sequence
\[ \xymatrix{ \sfC \ar[r]<0.5ex> \ar@{<-}[r]<-0.5ex> &
\stmod (kG) \ar[r]<0.5ex> \ar@{<-}[r]<-0.5ex> & \sfD. }
\] 

In this situation the Balmer spectrum of $\stmod (kG)$ is
disconnected, provided that $\sfC\neq 0\neq \sfD$. Recall
from \cite{BaSpec} that for any essentially small tensor triangulated
category $\sfT$ there is an associated spectral topological space
$\Spc \sfT$ which provides a universal notion of tensor-compatible
supports for objects of $\sfT$. The above localisation sequence yields
a decomposition
\[ \Spc (\stmod kG)= \Spc \sfC \coprod \Spc \sfD\] by
Theorem~\ref{th:disconnected}. We refer to Appendix~\ref{ap:appendix}
for further explanations.
\end{remark}

In the next section we describe an example that illustrates the
preceding remark.

\section{Thick subcategories}
\label{se:thick} Let $G$ be a finite group and fix a prime $p\in\bbZ$.

For each positive integer $n$, the canonical homomorphism
$\phi_{n}\colon \bbZ\to \bbZ/{p^n}$ is an epimorphism of rings, and
$\bbZ/{p^n}$ is finitely presented over $\bbZ$. Restriction via
$\phi_{n}$ identifies the $\bbZ/{p^n}G$-modules with all $\bbZ
G$-modules that are annihilated by $p^n$.

Let $\sfD_n$ be the full subcategory of $\stmod(\bbZ G)$ whose objects
are the $\bbZ G$-modules isomorphic to modules annihilated by $p^n$;
this is a tensor ideal thick subcategory and restriction via
$\phi_{n}$ identifies it with $\stmod (\bbZ/{p^n}G)$. Note that any
$kG$-module $M$ is in $\sfD_n$ when $M\otimes_{\bbZ}M$ is in
$\sfD_{n}$. This is easy to verify and means that $\sfD_{n}$ is
radical in the sense of Balmer~\cite{BaSpec}.

When $p$ does not divide $|G|$, one has $\sfD_{n}=0$ for each $n$, by
Theorem~\ref{th:local}.

\begin{proposition}
\label{pr:thick-tower} 
When $p$ divides $|G|$ one has $\sfD_n\neq \sfD_{n+1}$ for each $n\geq
1$, so these tensor ideal thick subcategories form a strictly
increasing chain:
\[ 
\sfD_1\subsetneq \sfD_2\subsetneq \sfD_3\subsetneq \ldots\subseteq \stmod (\bbZ G)\,.
\] 
Moreover, $\bbZ\not\in \sfD_{n}$ for any $n$, so $\bigcup_{n\geqslant
  1}D_{n}$ is a proper thick subcategory of $\stmod (\bbZ G)$.
\end{proposition}

\begin{proof} Fix an $n\geq 1$. Evidently $\bbZ/p^{n+1}$ is in
$\sfD_{n+1}$; we claim that it is not in $\sfD_{n}$. Indeed, the
inclusion $\sfD_n\to \stmod (\bbZ G)$ admits
$\bbZ/{p^n}\otimes_\bbZ -$ as a left adjoint; see
Remark~\ref{rem:epimorphisms}. It thus suffices to show that the
natural morphism $\eps\colon \bbZ/{p^{n+1}}\to \bbZ/{p^{n}}$ is not an
isomorphism in $\stmod (\bbZ G)$. Assume to the contrary that it admits
an inverse, say a morphism $\sigma\colon \bbZ/{p^{n}}\to
\bbZ/{p^{n+1}}$. Observe that $\sigma(1)=ap$, for some integer $a$,
since $\sigma$ is to be $\bbZ$-linear.

In $\stmod (\bbZ G)$ the morphism $\eps\sigma$ is the identity map of
$\bbZ/{p^{n}}$, so that in $\mod (kG)$ the endomorphism $\id -
\eps\sigma$ of $\bbZ/p^{n}$ admits a factorisation
\[ 
\xymatrixcolsep{1pc} 
\xymatrixrowsep{1pc} 
\xymatrix{ \bbZ/{p^{n}}
\ar@{->}[dr]_{\kappa} \ar@{->}[rr]^{\id - \eps\sigma} &&\bbZ/{p^{n}}
\\ &(\bbZ/{p^{n}})G \ar@{->}[ur]_{\pi}&}
\] where $\pi$ is the natural morphism; see
Theorem~\ref{th:trace}. Write $\kappa(1)=\sum_{g\in G}g\otimes c_{g}$
with $c_{g}\in \bbZ/{p^{n}}$. Note that $c_{g}=c_{1}$ for each $g$,
since $\kappa$ is $\bbZ G$-linear.  Thus the commutativity of the
diagram above yields
\[ 
1 - ap = (\id - \eps\sigma)(1) = \pi(\sum_{g\in G}g\otimes c_{1}) =
|G|c_{1} \quad \text{in $\bbZ/{p^{n}}$.}
\] This is a contradiction, for $p$ divides $|G|$. Thus,
$\sfD_{n}\ne\sfD_{n+1}$, as claimed.

It remains to verify that $\bbZ$ is not in $\sfD_{n}$ for any $n\geq
1$. As above, it suffices to verify that the natural morphism $\bbZ\to
\bbZ/p^{n}$ is not an isomorphism in $\stmod (\bbZ G)$, but this is clear
because $\Hom_{\bbZ}(\bbZ/p^{n},\bbZ)=0$.
\end{proof}

In order to connect the preceding computation to the problem of
classifying thick subcategories of the stable module category, we
recall some basics concerning relative cohomology; see
\cite[\S3.9]{Benson:1991a}.

\subsection*{Relative cohomology}
Let $k$ be a commutative ring and $G$ a finite group. A \emph{relative
  projective resolution} of a $kG$-module $M$ is a complex $P$ of
$kG$-modules with each $P^{i}$ weakly projective, $P^{i}=0$ when
$i>0$, and equipped with a morphism of complexes $\eps\colon P\to M$
of $kG$-modules such that the augmented complex
\[
\cdots \lto P^{-1}\lto P^{0}\xra{\ \eps\ } M\lto 0
\]
is $k$-split exact. Relative projective resolutions exists and are
unique up to homotopy; they can be constructed starting with the
morphism $\pi_{M}\col M\ua^{G}\to M$. For any such $P$, and
$kG$-module $N$, we set
\begin{equation}\label{eq:ext}
\Ext_{G,1}^{*}(M,N) = H^{*}(\Hom_{kG}(P,N))\,.
\end{equation}
These $k$-modules are independent of the choice of $P$. Note that 
\[
\Ext_{G,1}^{*}(M,N) \cong \Ext_{kG}^{*}(M,N)
\]
when $M$ is projective over $k$, since any $kG$-projective resolution of $M$ is  a relative projective resolution. In particular, $\Ext_{G,1}^{*}(k,k)\cong \Ext_{kG}^{*}(k,k)$.

As usual, $\Ext_{G,1}^{*}(M,M)$ is a graded $k$-algebra, and $\Ext_{G,1}^{*}(M,N)$ is a graded bimodule with $\Ext_{G,1}^{*}(N,N)$ acting on the left and $\Ext_{G,1}^{*}(M,M)$ on the right. The functor $-\otimes_{k}M$, with diagonal $G$-action, induces a homomorphism
\[
\zeta_{M}\colon \Ext_{G,1}^{*}(k,k)\to \Ext_{G,1}^{*}(M,M)
\]
of graded $k$-algebras, and the $\Ext_{G,1}^{*}(k,k)$ action on $\Ext_{G,1}^{*}(M,N)$ through $\zeta_{M}$ and through $\zeta_{N}$ coincide, up to the usual sign.

For a $kG$-module $M$ we set
\[
H^{*}(G,M)=\Ext_{G,1}^{*}(k,M)\,.
\]
The $k$-algebra $H^{*}(G,k)$ is graded-commutative and $H^{*}(G,M)$ is a graded module over it. Observe that for each $kG$-module $N$, there is a natural isomorphism of graded $H^{*}(G,k)$ modules:
\begin{equation}
\label{eq:relasgc}
\Ext^{*}_{G,1}(M,N) \cong H^{*}(G,\Hom_{k}(M,N))\,.
\end{equation}
Indeed, let $P$ be a $kG$-projective resolution of $k$, so that
$P\otimes_{k}M$ with diagonal $G$-action is a relative projective
resolution of $M$. The isomorphism above is induced by the adjunction
isomorphism
\begin{align*}
\Hom_{kG}(P\otimes_{k}M,N) \cong \Hom_{kG}(P,\Hom_{k}(M,N))\,.
\end{align*}

\subsection*{Supports} 
Let now $k$ be a noetherian commutative ring and $M$ a
$kG$-module. The ring $H^{*}(G,k)$ is then finitely
generated as a $k$-algebra, hence a noetherian ring, and the
$H^{*}(G,k)$-module $H^{*}(G,M)$ is finitely generated whenever $M$ is
finitely generated; see Evens~\cite{Evens:1961a}. From
\eqref{eq:relasgc} one then deduces that $\Ext^{*}_{G,1}(M,N)$ is
finitely generated over $H^{*}(G,k)$ for any pair of finitely
generated $kG$-modules $M$ and $N$.

We write $\mcV_{G}$ for the homogeneous prime ideals of $H^{*}(G,k)$
not containing the ideal of elements of positive degree, $H^{\geqslant
  1}(G,k)$, and for each finitely generated $kG$-module $M$, set
\[
\mcV_G(M) = \{\fp\in \mcV_{G} \mid \Ext^{*}_{G,1}(M,M)_{\fp}\ne 0\}\,.
\] 
This is a closed subset of $\mcV_{G}$ that we call the
\emph{support of $M$}. It coincides with 
\[
\{\fp\in \mcV_{G} \mid \Ext^{*}_{G,1}(M,N)_{\fp}\ne 0 \text{ for some
  $N\in\mod kG$}\}
\]
since the action of $H^*(G,k)$ on $ \Ext^{*}_{G,1}(M,N)$ factors
through that of $\Ext^{*}_{G,1}(M,M)$.

To formulate the next result, we recall that $k$ is \emph{regular} if it is noetherian and each finitely generated $k$-module has a finite projective resolution. For any $k$-module $M$, we write $\Supp_{k}M$ for the subset $\{\fp\in\Spec k\mid M_\fp\ne 0\}$ of $\Spec k$.

\begin{proposition}
\label{prop:supp-gd}
Let $k$ be a regular ring. Given finitely generated $k$-modules $M,N$ with trivial
$G$-action, if $\Supp_{k}M\subseteq\Supp_{k}N$, then $\mcV_G(M)\subseteq \mcV_G(N)$.
\end{proposition}

\begin{proof}
It is not hard to verify that 
\[
\Supp_{k}M= \Supp_{k}\End_{k}(M)\,.
\]
Indeed, as $M$ is finitely generated localisation at a prime $\fp$ induces an isomorphism
\[
\End_{k}(M)_{\fp}\cong \End_{k_{\fp}}(M_{\fp})\,.
\]
Thus, after localisation, we have to verify $M=0$ if and only if $\End_{k}(M)=0$, and this is clear.

Consequently, if $\Supp_{k}M\subseteq\Supp_{k}N$, then $\Supp_{k}\End_{k}(M)\subseteq \Supp_{k}\End_{k}(N)$. Then Hopkins' theorem~\cite[Theorem~11]{Hopkins:1987a} (see also \cite[Theorem~1.5]{Neeman:1992a}) implies that $\End_{k}(M)$ is in the thick subcategory generated by $\End_{k}(N)$, in the derived category of $k$-modules, and hence in the derived category of $kG$-modules; this is where we need the trivial action. It follows from \eqref{eq:relasgc} that $\mcV_G(M)\subseteq \mcV_G(N)$.
\end{proof}

\begin{remark}
  \label{re:bcr-not} 
  When $k$ is a field, assigning a subcategory $\sfC$ of $\stmod (kG)$
  to its support, $\bigcup_{M\in\sfC}\mcV_{G}(M)$, induces a bijection
  between the tensor ideal thick subcategories and the specialisation
  closed subsets of $\mcV_{G}$. This result is due to Benson, Carlson,
  and Rickard~\cite[Theorem~3.4]{Benson/Carlson/Rickard:1997a} for $k$
  algebraically closed, and proved in
  \cite[Theorem~11.4]{Benson/Iyengar/Krause:2011b} in general. It
  follows from Proposition~\ref{pr:thick-tower} that such a
  classification result cannot hold when $k=\bbZ$: Pick a prime $p$
  dividing $|G|$; then $\bbZ/p^{n}$ is not in the tensor ideal thick
  subcategory of $\stmod (\bbZ G)$ generated by $\bbZ/p$ when
  $n>1$. However, we claim that there are equalities
\[ 
\mcV_G(\bbZ/p^{n}) = \mcV_G(\bbZ/p)\quad\text{for all $n\geq 1$.}
\] 
This follows from Proposition~\ref{prop:supp-gd}, since
$\Supp_{\bbZ}(\bbZ/p^{n})=\{(p)\}$ for each $n\geq 1$.
\end{remark}

\section{The homotopy category of weakly injective modules}
\label{se:htpy_category}

In this section we study the homotopy category of weakly injective
modules and a pure variant.  As before, $k$ will be a commutative ring
and $G$ a finite group. An important ingredient of this discussion is
an additional exact structure for the category of $kG$-modules.

Let $\WInj kG$ denote the full subcategory of weakly injective
(equivalently, weakly projective) $kG$-modules, and $\winj kG$ is the
full subcategory of finitely presented modules in $\WInj kG$.

\subsection*{Exact structures} 

Recall that a sequence $0\to M'\to M\to M''\to 0$ in $\Mod (kG)$ is
$k$-split exact if it is split exact when restricted to the trivial
subgroup.  The sequence is \emph{$k$-pure exact} if it is
pure exact when restricted to the trivial subgroup.

\begin{lemma} 
For an exact sequence  $\eta\colon 0\to M'\to M\to M''\to 0$ of
$kG$-modules, the following conditions are equivalent:
\begin{enumerate}
\item The sequence is $k$-pure exact.
\item For every finitely presented weakly injective $kG$-module $P$,
each morphism $P\to M''$ lifts to a morphism $P\to M$.
\end{enumerate}
\end{lemma}
\begin{proof} 
  Recall from Definition~\ref{defn:purity} that $\eta$ is a pure exact
  sequence of $kG$-modules if and only if for every finitely presented
  $kG$-module $Q$, each morphism $Q\to M''$ lifts to a morphism $Q\to
  M$.
  Next observe that each finitely presented weakly injective
  $kG$-module is a direct summand of $Q\ua^G=kG\otimes_k Q$ for some finitely
  presented $k$-module $Q$. Applying 
\[\Hom_{kG}(kG\otimes_k Q,-)\cong\Hom_k(Q,-)\]
to the morphism $M\to M''$ yields the assertion.
\end{proof}

For any additive category $\sfC$ we write $\sfK(\sfC)$ for the
homotopy category of cochain complexes in $\sfC$. As usual,
$\sfK^{b}(\sfC)$ is the full subcategory of bounded complexes.

For a specified exact structure on $\sfC$, a complex $X=(X^n,d^n)$ is
said to be \emph{acyclic} if there are exact sequences
\[ 
0\to Z^n\xra{u^n} X^n\xra{v^n} Z^{n+1}\to 0
\] 
in $\sfC$ such that $d^n=u^{n+1}v^n$ for all $n\in\mathbb Z$. We
say  $X$ is \emph{acyclic in almost all degrees} if such exact
sequences exist for almost all $n$.

\subsection*{The bounded derived category}

We consider the category of finitely presented $kG$-modules with the
$k$-split exact structure and define the corresponding \emph{bounded
derived category} as the Verdier quotient
\[ \sfD^b(\mod kG)=\sfK^b(\mod kG)/\sfK_\ac^b(\mod kG),
\] where $\sfK_\ac^b(\mod kG)$ is the full subcategory of acyclic
complexes, with respect to the $k$-split structure, in $\sfK^b(\mod
kG)$.  The natural functor
\begin{equation}
\label{eq:D^b} \sfK^{+,b}(\winj kG)\lto \sfD^b(\mod kG)
\end{equation} is an equivalence of triangulated categories; its
inverse identifies a complex with its weakly injective
resolution. Here $\sfK^{+,b}(\winj kG)$ is the full subcategory of
$\sfK(\winj kG)$ of complexes $X$ with $X^{n}=0$ for
$n\ll 0$ and $X$ acyclic in almost all degrees.

\begin{proposition}
\label{pr:stmod} The composite functor
\[ \mod (kG)\rightarrowtail\sfD^b(\mod kG)\twoheadrightarrow\sfD^b(\mod
kG)/\sfK^b(\winj kG)
\] induces an equivalence of triangulated categories
\[ \stmod (kG)\stackrel{\sim}\lto\sfD^b(\mod kG)/\sfK^b(\winj kG).
\]
\end{proposition}

\begin{proof} Adapt the proof for group algebras over a field; see
Theorem~2.1 in \cite{Rickard:1989a}.
\end{proof}

Given bounded complexes $X,Y$ of finitely presented $kG$-modules and $n\in\bbZ$, set
\[\Ext^n_{G,1}(X,Y)=\Hom_{\sfD^b(\mod kG)}(X,Y[n]).\] This notation is
consistent with our previous definition \eqref{eq:ext}.

\subsection*{The pure derived category}

The \emph{pure derived category} of $\Mod (kG)$ is by definition the
following Verdier quotient
\[ \sfD_\pur(\Mod kG)=\sfK(\Mod kG)/\sfK_{\kpac}(\Mod kG),
\] where $\sfK_{\kpac}(\Mod kG)$ denotes the full subcategory
consisting of complexes that are acyclic with respect to the $k$-pure
exact structure.

\begin{proposition}\label{pr:derived} The derived category
$\sfD_\pur(\Mod kG)$ is a compactly generated triangulated
category. The inclusion $\winj kG\to \Mod (kG)$ induces an equivalence
of triangulated categories between $\sfK^b(\winj kG)$ and the full
subcategory of compact objects in $\sfD_\pur(\Mod kG)$.
\end{proposition}

\begin{proof} Consider the triangulated category $\sfK(\Mod kG)$ and
identify each $kG$-module with the corresponding complex concentrated
in degree zero.  Then each object in $\winj kG$ is compact in
$\sfK(\Mod kG)$ and
\begin{equation*} (\winj kG)^\perp=\sfK_{\kpac}(\Mod kG).
\end{equation*} Now apply Proposition~\ref{pr:compact}.
\end{proof}

\begin{lemma}
\label{le:derived-inj} The composite $\sfK(\WInj
kG)\rightarrowtail\sfK(\Mod kG)\twoheadrightarrow\sfD_\pur(\Mod kG)$
induces an equivalence of triangulated categories
\[ \sfK(\WInj kG)/\sfK_{\kpac}(\WInj kG)\stackrel{\sim}\lto
\sfD_\pur(\Mod kG).
\]
\end{lemma}

\begin{proof} This is an immediate consequence of
Corollary~\ref{cor:compact}.
\end{proof}

\begin{remark} 
  When $k$ is a field, every exact sequence of $k$-modules is pure
  exact, so $\sfD_\pur(\Mod kG)$ is the usual derived category of
  $kG$-modules. On the other extreme, if $G$ is trivial,
  $\sfD_\pur(\Mod kG)$ is the pure derived category of $\Mod (k)$
  studied in \cite{Krause:2011a}.
\end{remark}

\subsection*{The pure homotopy category}

The \emph{pure homotopy category} of $\WInj kG$ is by definition the
following Verdier quotient
\[ \sfK_\pur(\WInj kG)=\sfK(\WInj kG)/\sfK_\pac(\WInj kG)
\] where $\sfK_{\pac}(\WInj kG)$ denotes the full subcategory
consisting of complexes that are acyclic with respect to the pure
exact structure.

\begin{proposition}
\label{pr:KInj} The homotopy category $\sfK_\pur(\WInj kG)$ is
compactly generated and the natural functor \eqref{eq:D^b} induces an
equivalence of triangulated categories between the full subcategory of
compact objects in $\sfK_\pur(\WInj kG)$ and $\sfD^b(\mod kG)$.
\end{proposition}

\begin{proof} Choose for each $M\in\mod (kG)$ a weakly injective
resolution $\bsi M$. Arguing as in the proof of
\cite[Lemma~2.1]{Krause:2005a}, one can verify that the morphism $M\to
\bsi M$ induces, for each $X\in\sfK(\WInj kG)$, an isomorphism
\[ \Hom_{\sfK(\Mod kG)}(\bsi M,X)\xra{\sim}\Hom_{\sfK(\Mod
kG)}(M,X)\,.
\] Thus each $\bsi M$ is compact in $\sfK(\WInj kG)$ and
\begin{equation*} \{\bsi M\mid M\in\mod (kG)\}^\perp=\sfK_\pac(\WInj
kG).
\end{equation*} Now apply Proposition~\ref{pr:compact}.
\end{proof}

\subsection*{Recollements}

A \emph{recollement} is by definition a diagram of exact functors
\begin{equation*} \xymatrix{\sfT'\,\ar[rr]|-I&&\,\sfT\,
\ar[rr]|-Q\ar@<1.25ex>[ll]^-{I_\lambda}\ar@<-1.25ex>[ll]_-{I_\rho}&&
\,\sfT''\ar@<1.25ex>[ll]^-{Q_\lambda}\ar@<-1.25ex>[ll]_-{Q_\rho} }
\end{equation*} satisfying the following conditions, where $\sfS$
denotes the full subcategory consisting of objects $X\in\sfT$
satisfying $Q(X)=0$.
\begin{enumerate}
\item The functors $(I_\lambda,I)$, $(I,I_\rho)$, $(Q_\lambda,Q)$,
$(Q,Q_\rho)$ form adjoint pairs.
\item The functor $I$ induces an equivalence $\sfT'\xra{\sim}\sfS$.
\item The functor $Q$ induces an equivalence
$\sfT/\sfS\xra{\sim}\sfT''$.
\end{enumerate}

\begin{proposition}
\label{pr:recoll} Let $\sfT$ be a triangulated category with
set-indexed products and coproducts.  Fix a pair of essentially small
thick subcategories $\sfC$ and $\sfD$ consisting of compact objects
such that $\sfD\subseteq\sfC$. Then there is an induced recollement of
compactly generated triangulated categories
  \[
\xymatrix{(\sfD^\perp)/(\sfC^\perp)\,\ar[rr]|-I&&\,\sfT/(\sfC^\perp)\,
\ar[rr]|-Q\ar@<1.25ex>[ll]\ar@<-1.25ex>[ll]&&
\,\sfT/(\sfD^\perp)\ar@<1.25ex>[ll]\ar@<-1.25ex>[ll] }.
 \] Restricting the left adjoints of $I$ and $Q$ to the full
subcategories of compact objects yields up to equivalence and
idempotent completion an sequence of exact functors:
\[ \xymatrix{\sfC/\sfD\ar@{<<-}[r]&\sfC\ar@{<-<}[r]&\,\sfD}.
\]
\end{proposition}

\begin{proof} We have a pair of localising subcategories
$\sfC^\perp\subseteq\sfD^\perp$ which induce a quotient functor
$\sfT/(\sfC^\perp)\twoheadrightarrow \sfT/(\sfD^\perp)$ with kernel
$(\sfD^\perp)/(\sfC^\perp)$; see Proposition~2.3.1 of Chap.~II in
\cite{Verdier:1996a}.  The categories $\sfC^\perp$ and $\sfD^\perp$
form subcategories of $\sfT$ that are closed under set-indexed
products and coproducts. Thus $I$ and $Q$ preserve set-indexed
products and coproducts. It follows from Brown representability that
$I$ and $Q$ both have left and right adjoints.

A left adjoint of a coproduct preserving functor preserves
compactness. The description of the categories of compact objects
follows from Proposition~\ref{pr:compact}.
\end{proof}

\subsection*{A recollement} We apply Proposition~\ref{pr:recoll}: Let
$\sfT=\sfK(\WInj kG)$ and set
\[ \sfC=\sfK^{+,b}(\winj kG)\quad\text{and}\quad \sfD=\sfK^b(\winj
kG)\,.
\] Note that
\[ \sfC^\perp=\sfK_\pac(\WInj kG)\quad\text{and}\quad
\sfD^\perp=\sfK_{\kpac}(\WInj kG)\,.
\] The inclusion $\sfK_\pac(\WInj kG)\rightarrowtail
\sfK_{\kpac}(\WInj kG)$ induces the following commutative diagram of
exact functors.
\[ \xymatrix{ \sfK_\pac(\WInj
  kG)\ar@{>->}[d]\ar@{=}[r]&\sfK_\pac(\WInj kG) \ar@{>->}[d]\\
  \sfK_{\kpac}(\WInj kG) \,\ar@{>->}[r]\ar@{->>}[d]& \sfK(\WInj kG)
  \ar@{->>}[r]\ar@{->>}[d]&\sfD_\pur(\Mod kG) \ar@{=}[d]\\ \StMod^\cw
  (kG)\, \ar@{>->}[r]^-I&\sfK_\pur(\WInj kG)
  \ar@{->>}[r]^-Q&\sfD_\pur(\Mod kG)}
\] The occurrence of the stable module category in this diagram is
explained by the following lemma.

\begin{lemma}
\label{le:StMod} Taking a complex of $kG$-modules to its cycles in
degree zero induces an equivalence of triangulated categories
\[ \sfK_{\kpac}(\WInj kG)/\sfK_\pac(\WInj
kG)\stackrel{\sim}\lto\StMod^\cw (kG).
\]
\end{lemma}

\begin{proof} Denote by $\sfK_\ac(\WInj kG)$ the category of complexes
that are acyclic with respect to the $k$-split exact structure. Taking
cycles in degree zero induces an equivalence of triangulated
categories
\[ \sfK_\ac(\WInj kG)\stackrel{\sim}\lto\StMod (kG).
\] This functor identifies $\sfK_\pac(\WInj kG)\cap \sfK_\ac(\WInj
kG)$ with the category of cw-injective $kG$-modules by
Lemma~\ref{le:perp}, and induces therefore an equivalence
\[ \sfK_{\ac}(\WInj kG)/(\sfK_\pac(\WInj kG)\cap \sfK_\ac(\WInj
kG))\stackrel{\sim}\lto\StMod^\cw (kG).
\] Now observe that
\[\sfK_\ac(\WInj kG)\subseteq \sfK_{\kpac}(\WInj kG)
\] and apply Corollary~\ref{cor:compact}.
\end{proof}

The following result establishes the analogue of Proposition~6.1 in
\cite{Benson/Krause:2008a} for general rings of coefficients.

\begin{theorem} The functors $I$ and $Q$ induce a recollement
\[ \xymatrix{\StMod^\cw (kG)\,\ar[rr]|-I&&\,\sfK_\pur(\WInj kG)\,
\ar[rr]|-Q\ar@<1.25ex>[ll]\ar@<-1.25ex>[ll]&& \,\sfD_\pur(\Mod
kG)\ar@<1.25ex>[ll]\ar@<-1.25ex>[ll] }.
\] Restricting the left adjoints of $I$ and $Q$ to the full
subcategories of compact objects yields up to equivalence the
following sequence
\[ \xymatrix{ \stmod (kG)\ar@{<<-}[r]&\sfD^b(\mod kG)
\ar@{<-<}[r]&\,\sfK^b(\winj kG).}
\]
\end{theorem}

\begin{proof} We apply Proposition~\ref{pr:recoll}.  The functor $Q$
is a quotient functor by Lemma~\ref{le:derived-inj}, and its kernel
identifies with $\sfK_{\kpac}(\WInj kG)/\sfK_\pac(\WInj kG)$; see
Proposition~2.3.1 of Chap.~II in \cite{Verdier:1996a}.  The
description of the kernel as stable module category follows from
Lemma~\ref{le:StMod}.  The description of the compact objects follows
from Propositions~\ref{pr:stmod}, \ref{pr:derived}, and \ref{pr:KInj}.
\end{proof}

\begin{remark} 
\label{rem:adjoints}
The adjoints of $I$ and $Q$ in this recollement admit explicit
descriptions. More precisely, $I_\lambda=\bst k\otimes_k-$,
$I_\rho=\Hom_k(\bst k,-)$, $Q_\lambda=\bsp k\otimes_k-$,
$Q_\rho=\Hom_k(\bsp k,-)$.  Here, we write $\bsp k$ for a weakly
projective resolution, and $\bst k$ for a Tate resolution of the
trivial $kG$-module $k$; see \cite[\S6]{Benson/Krause:2008a}.
\end{remark}

\appendix

\section{Disconnecting spectra via localisation sequences}
\label{ap:appendix}
\begin{center}
{\sc by  Greg Stevenson}
\end{center}

Let $\sfS$ be an essentially small tensor triangulated category, for
instance the stable module category $\stmod (kG)$ of a group algebra
$kG$. In \cite{BaSpec} Balmer has associated to $\sfS$ a spectral
topological space $\Spc \sfS$ which provides a universal notion of
tensor-compatible supports for objects of $\sfS$. We briefly recall
the definition of $\Spc \sfS$. A proper tensor ideal (our tensor
ideals are always assumed to be thick) $\sfP$ of $\sfS$ is
\emph{prime} if $s\otimes s' \in \sfP$ implies $s\in \sfP$ or $s'\in
\sfP$ for all $s,s' \in \sfS$. The \emph{spectrum} of $\sfS$ is then
defined to be the set $\Spc \sfS$ of prime tensor ideals of $\sfS$
endowed with the \emph{Zariski topology} which is given by the basis
of closed subsets
\[ \{\supp_\sfS s = \{\sfP \in \Spc \sfS \mid s\notin \sfP\} \mid s\in
\sfS\}.
\] The subset $\supp_\sfS s$ is called the \emph{support} of the
object $s$. More generally one can define the support of any
subcategory $\sfC \subseteq \sfS$ by
\[ \supp_\sfS \sfC = \bigcup_{c\in \sfC} \supp_\sfS c.
\] The spectrum is contravariantly functorial with respect to strong
monoidal exact functors \cite[Proposition~3.6]{BaSpec} i.e., given a
strong monoidal exact functor $i_*\colon\sfR \to \sfS$ there is a
spectral map $\Spc(i_*)\colon \Spc \sfS \to \Spc \sfR$.

We show that if a tensor triangulated category $\sfS$ admits a
semiorthogonal decomposition given by a pair of tensor ideals then the
spectrum of $\sfS$ is disconnected. In the remainder of this section,
$\sfS$ is an essentially small tensor triangulated category, with unit
object $\one$, and
\[ \xymatrix{ \sfR \ar[r]<0.5ex>^-{i_*} \ar@{<-}[r]<-0.5ex>_-{i^!} &
\sfS \ar[r]<0.5ex>^-{j^*} \ar@{<-}[r]<-0.5ex>_-{j_*} & \sfT }
\] is a localisation sequence such that both $\sfR $ and $\sfT $ are
tensor ideals of $\sfS$, when viewed as subcategories via the
embeddings $i_*$ and $j_*$. We recall that this diagram being a
\emph{localisation sequence}, or \emph{semiorthogonal decomposition},
means that $i_*$ and $j_*$ are fully faithful exact functors such that
$i^!$ is right adjoint to $i_*$ and $j^*$ is left adjoint to
$j_*$. Moreover, $(i_*\sfR)^\perp=j_*\sfT$ and
$i_*\sfR={^\perp(j_*\sfT)}$.

In general one must ask that both $i_*\sfR$ and $j_*\sfT$ are tensor
ideals; $i_*\sfR$ being an ideal does not necessarily imply $j_*\sfT$
is an ideal and vice versa. However, there is a special case in which
these implications hold and we feel it is worth making explicit. We
say the tensor triangulated category $\sfS$ is \emph{rigid} if the
symmetric monoidal structure on $\sfS$ is closed, with internal hom
$\fHom(-,-)$ and, setting $s^\vee = \fHom(s,\one)$ for $s\in
\sfS$, the natural morphism
\[ s^\vee \otimes s' \to \fHom(s,s')
\] is an isomorphism for all $s,s'\in \sfS$. For further details the
interested reader can consult Section~1 of \cite{Deligne/Milne:1982}
in the general case and Appendix~A of \cite{HPS} for further details
in the triangulated case.

\begin{lemma}
\label{lem_rigid_ideal} Let $\sfS$ be an essentially small rigid
tensor triangulated category. If $\sfR$ is a tensor ideal of $\sfS$,
then so are the thick subcategories $\sfR ^{\perp}$ and
${}^\perp\sfR$.
\end{lemma}

\begin{proof} It is clear from the adjunction between the tensor
product and internal hom that $\sfR^\perp$ is closed under the
functors $\fHom(s,-)$ for $s\in \sfS$. As $\fHom(s,-) \cong
s^\vee\otimes(-)$ and the duality $(-)^\vee$ is a contravariant
equivalence it is thus immediate that $\sfR^\perp$ is an ideal as
claimed. The proof that ${}^\perp\sfR$ is an ideal is similarly
trivial.
\end{proof}

We now return to the setting where $\sfS$ is only assumed to be tensor
triangulated and commence the proof of the promised result on
disconnectedness of the spectrum. Recall that an object $s\in \sfS$ is
said to be $\otimes$-\emph{idempotent} if $s\otimes s \cong s$.

\begin{lemma}
\label{lem_small_idempotents} There are $\otimes$-idempotent objects
$\gam_\sfR \one$ and $L_\sfR \one$ such that
\[ \gam_\sfR \one\otimes(-) \cong i_*i^! \quad \text{and} \quad
L_\sfR \one \otimes (-) \cong j_*j^*.
\] In particular, $\gam_\sfR \one \otimes L_\sfR \one
\cong 0$.
\end{lemma}

\begin{proof} This is standard, see for instance
\cite[Theorem~2.13]{BaRickard}
\end{proof}

\begin{lemma}
\label{lem_projformula} Given objects $r\in \sfR $, $s\in \sfS$, and
$t\in \sfT $ there are isomorphisms
\[ s\otimes i_*r \cong i_*(i^!s\otimes r) \quad \text{and} \quad
s\otimes j_*t \cong j_*(j^*s\otimes t).
\]
\end{lemma}

\begin{proof} The proof of the second isomorphism is almost exactly as
in \cite[Lemma~8.2]{StevensonActions} and the proof of the first is
similar.
\end{proof}

\begin{lemma} The map $\Spc(j^*)\colon\Spc \sfT \to \Spc \sfS$ gives a
homeomorphism onto $\supp_\sfS j_*\sfT $ and $\Spc(i^!)\colon \Spc
\sfR \to \Spc \sfS$ is a homeomorphism onto $\supp_\sfS i_*\sfR
$. Furthermore, there are equalities
\[ \supp_\sfS i_*\sfR = \supp_\sfS \gam_\sfR \one \quad
\text{and} \quad \supp_\sfS j_*\sfT = \supp_\sfS L_\sfR \one,
\] so that both subsets are closed in $\Spc \sfS$.
\end{lemma}

\begin{proof} The arguments for $\sfR $ and $\sfT $ are similar so we
only prove the result for $\sfT $. We already know from
\cite[Proposition~3.11]{BaSpec} that the localisation $j^*$ induces a
homeomorphism
\[ \Spc \sfT \stackrel{\sim}{\to} \{\sfP\in \Spc \sfS \; \vert
\; \sfR \subseteq \sfP\},
\] and it remains to identify this subset with the support of $\sfR
^{\perp}=j_*\sfT $.

First observe that since $\gam_\sfR \one \otimes L_\sfR
\one \cong 0$ every prime tensor ideal must contain one of these
idempotents. Furthermore, no proper ideal can contain both idempotents
as they build $\one$ via the localisation triangle
\[ \gam_\sfR\one \to \one \to L_\sfR\one \to \Sigma
\gam_\sfR\one.
\] So for $\sfP\in \Spc \sfS$ we have $\sfR \subseteq
\sfP$ if and only if $L_\sfR \one\notin \sfP$ if
and only if $j_*\sfT \nsubseteq \sfP$. This completes the
argument via the following trivial fact
\[ \supp_\sfS j_*\sfT = \{\sfP \in \Spc \sfS \; \vert \;
j_*\sfT \nsubseteq \sfP\}.
\] The final statement now follows as
\begin{align*} 
\supp_\sfS L_\sfR \one &= \{
\sfP \in \Spc \sfS \; \vert \; L_\sfR \one\notin
\sfP\} \\ &= \{ \sfP \in \Spc \sfS \; \vert \; \gam_\sfR
\one\in \sfP\} \\ &= \{ \sfP \in \Spc \sfS \;
\vert \; \sfR \subseteq \sfP\}\\ &= \Spc(j^*)(\Spc\sfT ) \\ &=
\supp_\sfS j_*\sfT. \qedhere
\end{align*}
\end{proof}

From this point onward let us be lazy and identify $\sfR $ and $\sfT $
with thick subcategories of $\sfS$ and view their spectra as subsets
of $\Spc \sfS$.

\begin{theorem}\label{th:disconnected}
The subsets $\Spc \sfR = \supp_\sfS \sfR $ and $\Spc
\sfT =\supp_\sfS \sfT $ of $\Spc \sfS$ are open and closed, and there
is a decomposition
\[ \Spc \sfS = \Spc \sfR \coprod \Spc \sfT.
\] In particular, if neither $\sfR $ nor $\sfT $ is the zero ideal the
space $\Spc \sfS$ is disconnected.
\end{theorem}

\begin{proof} By \cite[Proposition~3.11]{BaSpec} and the last lemma
respectively we can describe $\Spc \sfR $ as both
\[ \{\sfP\in \Spc \sfS \; \vert \; \sfT \subseteq \sfP\}
\quad \text{and} \quad \supp_\sfS\sfR .
\] It is clear from the first description that $\Spc \sfR $ is
generalisation closed and from the second that $\Spc \sfR $ is
specialisation closed. By the last lemma this subset is constructible
(since it is the support of an object) and hence it is both closed and
open. The argument for $\sfT $ is the same. The existence of the
claimed decomposition is clear: by Lemma \ref{lem_small_idempotents}
we have $\otimes$-orthogonal Rickard idempotents, precisely one of
which must be contained in any non-zero proper prime ideal (this can
also be used to argue that both subsets in question are open).
\end{proof}

For completeness we give an easier proof of a stronger result in the
rigid case.

\begin{proposition} Suppose $\sfS$ is as above but is furthermore
assumed to be rigid. Then $\sfS$ is equivalent to $\sfR \oplus
\sfT$. In particular the spectrum decomposes as in the above theorem.
\end{proposition}

\begin{proof} Using rigidity of $\sfS$ we have, for $i\in \bbZ$,
\[ 
\Hom(L_\sfR\one, \Sigma^i \gam_\sfR\one) \cong
\Hom(\Sigma^{-i}\one, (L_\sfR\one)^\vee \otimes
\gam_\sfR\one) \cong 0,
\] 
the final isomorphism since thick tensor ideals are closed under
$(-)^\vee$ and $\sfR \cap \sfT = 0$. Hence the localisation triangle
\[ 
\gam_\sfR\one \to \one \to L_\sfR\one \to \Sigma
\gam_\sfR\one
\] 
splits yielding $\one \cong \gam_\sfR\one \oplus L_\sfR\one$. 
It is now clear that $\sfS$ splits.
\end{proof}

\begin{corollary}\label{co:disconnected}
 If $\sfS$ has a decomposition, as in the assumptions
of this section, given by non-zero tensor ideals and $\one$ is
indecomposable then $\sfS$ is not rigid.\qed
\end{corollary}

\bibliographystyle{amsplain}
\providecommand{\bysame}{\leavevmode\hbox to3em{\hrulefill}\thinspace}

\end{document}